\documentclass[12pt]{amsart}
\usepackage{amssymb}

\newtheorem{theorem}{Theorem}[section]

\newtheorem{lemma}{Lemma}[section]

\newtheorem{conj}{Conjecture}[section]

\title[Arithemetic properties of blocks of consecutive integers]
{Arithmetic properties of blocks of consecutive integers}

\author{Tarlok N. Shorey}
\address{Department of Mathematics\\
Indian Institute of Technology\\
Powai, Mumbai-400 076, India }
\email{shorey@math.iitb.ac.in}

\author{Rob Tijdeman}
\address{Mathematical Institute\\
Leiden University\\
2300 RA Leiden, P.O. Box 9512\\
The Netherlands}
\email{tijdeman@math.leidenuniv.nl}



\thanks{The authors are grateful for the hospitality of the Max-Planck-Institute for Marthematics in Bonn, Germany. A basis of this paper was laid during a stay of the first author from May until July, 2014, and of the second author in July.}

\begin{document}
\begin{abstract}
This paper provides a survey of results on the greatest prime factor, the number of distinct prime factors, the greatest squarefree factor and the greatest $m$-th powerfree part of a block of consecutive integers, both without any assumption and under assumption of the $abc$-conjecture. Finally we prove that the explicit $abc$-conjecture implies the Erd\H{o}s-Woods conjecture for each $k \geq 3$.
\end{abstract}

\maketitle

\section{Introduction} 
Arithmetic properties of blocks of integers play an important role in various number theoretic questions. Let $n$ and $k$ be integers with $n >k \geq 3$ and consider $N = n(n+1) \cdots (n+k-1)$. Lower bounds on the greatest prime factor and the number of distinct prime divisors of $N$ are used in results on Diophantine equations of the form $N=y^q$ where $y>1, q>1$ are integers (see e.g.\cite{ns},\cite{gss},\cite{lai}) and on  irreducibility of polynomials (see e.g. \cite{ft},\cite{fs},\cite{st10}).  It is a question of Erd\H{o}s and Graham \cite{eg}, p. 67, which disjoint blocks of positive integers each of length $\geq 4$ have a perfect power as the product of their terms. The question was answered in the negative for more than two blocks by Ulas \cite{ula} and Bauer and Bennett \cite{bb}.  See also \cite{lw},\cite{bl}. It happens for two blocks if and only if the squarefree parts of the products of the terms of both blocks are equal. Therefore estimates on the squarefree part of $N$ restrict the range of possible solutions. An open conjecture of Erd\H{o}s and Woods  \cite{erd80} states that there is a $k$ such that there are no positive integers $n_1<n_2$ such that the integers $n_1+i$ and $n_2+i$ have exactly the same prime divisors for $i=0,1, \dots,k-1$. In results on this question the greatest squarefree divisor plays a crucial role.

In this paper we give a survey of results on these functions in the literature and add some new results. Notation is introduced in Section 2. In Section 3 we summarize results on  the greatest prime factor of $N$, in Section 4 those on the number of distinct prime factors on $N$. Section 5 deals with the greatest squarefree divisor of $N$ and Section 6 with the greatest $m$-th powerfree part of $N$. Section 7 contains results on the problem of Erd\H{o}s and Woods. 

The $abc$-conjecture provides good insight in the true behaviour of the considered functions. In Section 8 we mention the results which have been obtained under asssumption of the $abc$-conjecture. In Section 9 we apply the explicit $abc$-conjecture to the Erd\H{o}s-Woods conjecture. Finally we refer in Section 10 to generalizations in the literature of the results mentioned in this survey.

New in this paper are a lower bound for the greatest prime factor of $N$ when $n$ is very large compared to $k$, see Theorem 3.1, an improved lower bound for the greatest $m$-th powerfree part of $N$ for given $k,m$, Theorem 6.1, another approach to prove results on the greatest $m$-th powerfree part of $N$ for given $k,m$ under the $abc$-conjecture, Theorem 8.1, a new estimate for the greatest prime factor of $N$ for general $n$ and $k$ under the $abc$-conjecture, Theorem 8.2, and the proof that the explicit $abc$-conjecture implies the Erd\H{o}s-Woods conjecture for each $k \geq 3$, Theorem 9.1.

\section{Notation}

We write $P(x)$ for the greatest prime factor of the integer $x$, 
$\omega(x)$ for the number of distinct prime factors of $x$,
$R(x)$ for the greatest squarefree factor (the `radical') of $x$ and $Q_m(x)$ for the greatest $m$-th powerfree part of $x$.
We put $P(1)=1, \omega(1)=0, R(1)=Q_m(1)=1.$ Let
\begin{equation} \label{delta}
N=n(n+1)\cdots(n+k-1) 
\end{equation}
and 
\begin{equation}
P(n,k) = P(N), \omega(n,k) = \omega(N), R(n,k) = R(N), Q_m(n,k) = Q_m(N).
\end{equation}
We write $\exp_2 x$ for $\exp \exp x$, $\log_2 x$ for $\log \log x$, and so on.
We use the Vinogradov symbols $\ll$ and $\gg$ where the attached parameters indicate on which the constants involved depend. In a similar way we use the Landau symbols $O(\cdot)$ and $o(\cdot)$.

\section{The greatest prime factor}
Here we only formulate the best available results. For the history and more details we refer to \cite{ls05} and Section 1.1 of \cite{st07}.

\subsection{Complete results}
Sylvester proved in 1892 that a product of $k$ consecutive positive integers greater than $k$ is divisible by a prime exceeding $k$. This has been refined by Laishram and Shorey \cite{ls05} to 
$$P(n,k)>1.8k  ~{\rm for}~ n>k$$
unless $(n,k)\in \{(8,3),(6,4),(7,4), (15,13), (16,13),(4,3),(5,4),(6,5),\\(9,8),
(12,11),(14,13),(15,14),(19,18),(64,63)\}$, to
$$P(n,k) > 1.97k~ {\rm for} ~n>k+13$$ and to $$P(n,k) > 2k ~{\rm for}~ n> \max(k+13, 279k/262).$$ Further they calculated all the exceptions to the inequality $P(n,k) > 1.95k$ with $n>k$, but this list is too long to be reproduced here. They noted that the restrictions on $n,k$ are necessary. For example, $P(279,262) \leq 2 \times 262.$

Nair and Shorey \cite{ns} proved that 
$$P(n,k)> 4.42 k  ~{ \rm for}~ n>4k$$
and that
$$P(n,k)> 4.5k  ~{\rm for}~ n>4k ~{\rm and}~ k \not= 82,83.$$
This implies that for $n>100$
$$P(n,k)>4.42k$$
whenever $n,n+1,\dots,n+k-1$ are all composite.

\subsection{Asymptotic results}
If $n \leq k^{1.9}$ and $n$ is sufficiently large, it follows from the results on the difference between consecutive primes (see e.g. \cite{bhp}) that there is a prime in $[n,n+k-1].$ Hence $P(n,k)\geq n$
whenever $n$ is sufficiently large and $n<k^{1.9}$

For $n<k^{2-\varepsilon}$ with $0 <\varepsilon \leq 0.1$ the best result up to now is due to C. Jia and Liu \cite{jl}, viz. 
$$P(n,k) \gg_{\varepsilon} k^{\frac{25}{13} - \varepsilon}.$$

Harman \cite{har} Ch. 6 proved  $$P(k^2,k) \gg k^{1.48}.$$ 

For $k^2 < n \leq \exp(c_1(\log k)^{3/2}/(\log_2 k)^{1/2})$ for suitable $c_1$ the best available result is due to Sander \cite{san}, viz.
$$P(n,k) \gg k^{1 + c_2(\log k / \log n)^2} $$
for some positive constant $c_2$. 

For $ \exp(c_1(\log k)^{3/2}/(\log_2 k)^{1/2}) < n \leq \exp_2 \left((\log_2 k)^3/( \log_3k)^2\right)$ the best available bound is due to Shorey \cite{sh74}, viz.
\begin{equation} \label{sh74}
 P(n,k) \gg k \log k \frac{\log_2k}{\log_3 k}.
 \end{equation}

For $\exp_2 \left((\log_2 k)^3/ \log_3k)^2\right) < n \leq \exp_2 (\log^2 k / \log_2 k)$ the best available bound is due to  Ramachandra \cite{ram}, viz.
\begin{equation} \label{ram}
P(n,k) \gg k \log k \left( \frac {\log_2 n}{\log_2 k}\right)^{1/2} .
\end{equation}

For $n > \exp_2 (\log^2 k / \log_2 k)$ this is superseded by a result of the authors \cite{st90}:
$$ P(n,k) \gg k \log_2 n.$$ 
Combining the above results we see that (\ref{sh74}) is valid for $n>k^{3/2}$.

According to Langevin \cite{lan} the constant involved in (\ref{ram}) can be arbitrarily close to 1 by letting $n$ tend to $\infty$.
Here we shall derive the following sharpening by using an improved linear form estimate due to Matveev \cite{mat}.
The proof is an extension of \cite{tij}, Section 10 where the case $k=2$ was treated.

\begin{theorem}
For $k \geq 2$, $n > \exp_2 k$ and $n$ sufficiently large we have
$$P(n,k) \gg k\log_2 n \frac {\log_3 n}{ \log_4 n}.$$
\end{theorem}

\begin{proof}
Suppose $N = p_1^{k_1} \cdots p_s^{k_s}$ with distinct primes $p_1, \dots p_s$ and positive integers $k_1, \dots, k_s$.
Let $P = P(n,k) = \max_{i=1, \dots s}p_i$. Suppose $\omega (n+h) = \min_{i= 0, \dots, k-1} \omega (n+i)$.
Let $n+h = p_1^{h_1} \cdots p_s^{h_s}$ and $n+j = p_1^{j_1} \cdots p_s^{j_s}$ for any $j \not= h, j \in \{0, \dots, k-1\}$. Then
\begin{equation} \label{est}
 0 < \left | \frac {p_1^{j_1} \cdots p_s^{j_s}} {p_1^{h_1} \cdots p_s^{h_s}} -1 \right | < \frac kn.
 \end{equation}
Set $s_i = \omega (n+i)$ for all $i$. Note that for each $j$ with $0 \leq j <k$ at most $\omega(n+j) + \omega(n+h) \leq 2s_j$ primes have nonzero exponents.
We apply the following special case of a linear form estimate.
\begin{lemma} [Matveev \cite{mat}]
Let $p_1, \dots, p_s$ be primes. Put $\Omega = \prod_{j=1}^s \log p_j.$
Then there exists an absolute constant $c_{3}$ such that the inequalities
$$0 < | p_1^{b_1} \cdots p_s^{b_s} -1 | < \exp( -e^{c_{3}s} \Omega \log B)$$
have no solution in rational integers $b_1,...,b_s$ of absolute values not exceeding $B \geq 2$.
\end{lemma}
\noindent This gives
\begin{equation} \label{mat}
 \left | \frac {p_1^{j_1} \cdots p_s^{j_s}} {p_1^{h_1} \cdots p_s^{h_s}} -1 \right | > \exp(-e^{2c_{3}s_j} \Omega \log B)
 \end{equation}
where $\Omega \leq ( \log P)^{2s_j}, B \leq \frac {\log(n+k)}{\log 2}$. 
It follows from (\ref{est}) and (\ref{mat}) that
$$ \frac {\log n}{\log_2 n} \ll \frac {\log(n/k)} {\log_2 n} \ll e^{2c_3s_j} (\log P)^{2s_j}.$$
This implies 
$$ \log_2n \ll s_j \log_2 P.$$
Without loss of generality we may assume $P \leq (\log_2 n)^3$. Hence, for $n$ sufficiently large, by $n > \exp_2 k$,
$$ s_j \gg \frac {\log_2 n } {\log_4 n} > 2 \pi(k).$$
It follows that $n+j$ for $j=0, \dots, k-1, j \not= h$ contains $ \gg \frac {\log_2 n } {\log_4 n}$ distinct prime factors $> k$.
Since two such numbers $n+j$ cannot both be divisible by the same prime $>k$, we find
$$ \omega(n,k) \gg k \frac {\log_2 n} { \log_4 n}.$$
Thus, by the Prime Number Theorem,
$$ P(n,k) \gg k \log_2n \frac {\log_3 n} { \log_4 n}.$$
\end{proof}

\section{The number of distinct prime factors}
\subsection{Complete results}
For more details we refer to \cite{lai}.
Sylvester's earlier mentioned result implies that $$\omega(n,k) > \pi(k)$$ if $n>k$.
Let $\delta (k) =2$ for $3 \leq k \leq 6, =1$ for $7 \leq k \leq 16$, and $=0$ otherwise.
Laishram and Shorey \cite{ls05} improved upon earlier results by specifying all the finitely many exceptions to the inequality $$\omega (n,k) \geq \pi(k) +\left[ \frac 34 \pi(k) \right] -1 + \delta(k).$$
As corollaries they obtained
$$ {\it If}~ (n,k) \notin \{(114,109), (114,113) \}, ~{\it then}~ \omega(n,k) \geq \pi(k) + \left[ \frac 23 \pi(k) \right] -1$$
and
$$\omega(n,k) \geq \min \left( \pi(k) +\left[ \frac 34 \pi(k) \right] -1 + \delta(k), \pi(2k)-1\right).$$
They further proved
$$ {\it If}~ (n,k) \not= (6,4)~ {\it and}~12n > 17k, ~{\it then}~ \omega(n,k) \geq \pi(2k).$$
Again this is best possible, in view of $\omega(6,4)= \pi(8)-1$ and $\omega(34,24)= \pi(48)-1$.


\subsection{Asymptotic results}
Ramachandra et al. \cite{rst} proved that there exists a constant $c_4>0$ such that for $n > \exp(c_4(\log k)^2)$,
$$\omega(n,k) > k.$$
Shorey and Tijdeman (\cite{st92}, Theorem 5) showed that there exists a constant $c_5>0$ such that for $k>c_5, n> \exp_2(c_5k)$,
$$ \omega(n,k) > k + \pi(k)-2.$$

\section{The greatest squarefree divisor}
There are hardly estimates on $R(n,k)$ in the literature, but the best approach seems to be the following.
Combine the results of the previous section with estimates on the product of the first primes.
For the complete results an explicit estimate is required, e.g.\\
{\it the product of the primes $\leq r$ is less than $3^r$}, see \cite{han}. \\
For the asymptotic results the best possible result reads \\
{\it the product of the primes $\leq r$ is $ e^{(1+o(1))r}$},\\
which follows from the proof of the Prime Number Theorem. It follows that the product of the first $t$ primes is $t^{(1+o(1))t}$. Thus, as an example, from the first result of subsection 4.2, we have for $n > \exp(c_4(\log k)^2)$,
$$R(n,k) \geq k^{ (1+o(1))k}.$$

Khod\v{z}aev \cite{kho} answered a question of Sprind\v{z}uk (\cite{spr} p.186) by showing that there exists no constant $c_6$ such that there are infinitely many pairs of integers $(n,k)$ with $k< (\log n)^{c_6}$ for which 
\begin{equation} \label{spr}
R(n,k) < k^k.
\end{equation}
More precisely, he showed that there exist positive constants $c_7, c_8$ such that (\ref{spr}) does not hold when $1 \leq k \leq c_7 \sqrt{n}$, but (\ref{spr}) holds for all $k \geq c_8 \sqrt{n}$.

\section{The $m$-th powerfree part}

Khod\v{z}aev \cite{kho2} proved the result on $Q_2(N)$ corresponding to his result on $R(N)$ mentioned above. Here $\sqrt{n}$ is replaced by $\sqrt{n/ \log n}$.

Sprind\v{z}uk \cite{spr} and Turk \cite{tu82} gave bounds for the greatest $m$-free part of a polynomial $f(x) \in \mathbb{Z}[x]$ at integers $x$. These bounds imply, for $k \geq 3$,
$$ Q_2(n,k) \gg_{\varepsilon,k} ( \log n)^{\frac{1}{12} -\varepsilon}.$$
 (cf. \cite{spr}, p. 161, formula (2.10)) and for $k \geq 3,m \geq 3$,
$$ Q_m(n,k) > c_9^{-m^2} (\log n)^{1/(2m^2 \phi(m))},$$
where $c_9>0$ depends only on $k$ and $\phi(m)$ is the Euler function
(cf. \cite{spr}, p. 139, formula (1.2)). 

It follows from a result of Turk \cite{tu82} on polynomial values that for $k\geq 3, m\geq 2$, $n > 20$ 
$$P(Q_m(n,k)) \gg_{k,m} \log_2 n$$
and for $k = 2, m \geq 3$,
$$P(Q_m(n,k)) \gg_m \log_3 n.$$

De Weger and Van de Woestijne \cite{ww} considered $$\lambda_{m}(n,k) = \max_{i=0,1, \dots, k-1} Q_m(n-i).$$ They proved that for $k \geq 2, m \geq 3$ and all $n \in \mathbb{N}$
\begin{equation} \label{wwo}
\lambda_m(n,k) \gg_{k,m} ( \log n)^{1/(2m-1)}.
\end{equation}
On the other hand, they showed that there are infinitely many pairs $k,m$ such that $$\lambda_m(n,k) \ll_{k,m} n^{1-m/(km-1)}.$$
If $k$ and $m$ are odd, the exponent may be replaced by $1-m/(km-2)$.
In Section 8 we argue that it is likely that the last two results are close to the best possible ones.

In order to obtain a lower bound for $Q_m(n,k)$ from (\ref{wwo}), let $j$ be the value of $i$ for which $\min_{i=0,1, \dots, k-1} Q_m(n-i)$ is attained. By the argument used to derive formula (\ref{wwo}) applied with $k=2$ to every pair $(n-i,n-j)$ with $i \not= j$ it follows that 
\begin{equation} \label{wwlike} \prod_{i=0}^{k-1} Q_{m}((n-i)(n-j)) \gg_{k,m} (\log n)^{(k-1)/(2m-1)} .
\end{equation}
We have to correct for the common factors. By an argument of Erd\H{o}s \cite{erd} it suffices to divide by 
\begin{equation} \label{erdos}
  \prod_{p <k} p^{[\frac {k}{p}] + [\frac {k} {p^2}] + \dots} <k!
 \end{equation}
which depends only on $k$. Therefore (\ref{wwlike}) and (\ref{erdos}) imply the following result.
\begin{theorem}
For $m \geq 3$ the greatest $m$-free part of $\prod_{i=0}^{k-1} (n+i)$ satisfies
$$ Q_m(n,k) \gg_{k,m} ( \log n)^{(k-1)/(2m-1)}.$$
\end{theorem}
\noindent The exponent of $\log n$ is much better than in the earlier mentioned result of Sprind\v{z}uk.

\section{The Erd\H{o}s - Woods conjecture}

Erd\H{o}s \cite{erd80}) conjectured that if $m,n$ are positive integers there exists a $k$ such that if for $1 \leq i \leq k$ the numbers $m+i$ and $n+i$ have the same prime factors, then $m=n$. This problem is equivalent with a problem of J. Robinson in logic (Woods \cite{woo}). The only known examples of positive integers $(m,n,k)$ with $1\leq m<n, k \geq 2$ and 
\begin{equation} \label{ew}
R(m+i) = R(n+i)~ {\rm for} ~ 0 \leq i <k
\end{equation}
arise from
$$R(2^h-2) = R(2^h(2^h-2)), R(2^h-1) = R((2^h-1)^2)~~(h \geq 2)$$
and
$$R(75) = R(1215) =15, R(76) = R(1216) = 38.$$

It is proved in \cite{bsw} that if (\ref{ew}) holds then
$$ \log k \ll ( \log m \log_2 m)^{1/2}~~{\rm for}~k \geq 3,$$
$$ \log (n-m) \gg k(\log k)^2 / \log_2 k ~~{\rm for}~k \geq 3,$$
and
$$\log(n-m) \gg k (\log k + \log_3 n) \log_2 y(\log_{3} y)^{-1}~~{\rm for}~k \geq 27.$$

\section{Results under the $abc$-conjecture}

The well known $abc$-conjecture reads as follows.

\begin{conj} \label{abc} {\rm [Oesterl\'e and Masser, 1985]}
For any given $\varepsilon > 0$ and coprime positive integers $a,b,c$ satisfying $a+b=c$ we have
$$c \ll_{\varepsilon} R(abc)^{1+\varepsilon}.$$
\end{conj}

\noindent In this section we give applications of this conjecture.

\subsection{Estimates for fixed $k$}
 Langevin \cite{lanc} (see also Granville \cite{gra}) proved that if the $abc$-conjecture is true and $g(x) \in \mathbb{Z}[x]$ has no repeated roots, then for any $\varepsilon >0$ we have 
$\prod_{p|g(n)} p \gg |n|^{{\rm deg}(g)-1 - \varepsilon}.$ It follows from this result that under the $abc$-conjecture 
\begin{equation} \label{lan}
 R(n,k) \gg_k |n|^{k-1- \varepsilon}.
 \end{equation}

De Weger and Van de Woestijne \cite{ww} showed that for any $\varepsilon >0, k \geq 2, m \geq 2$ under the $abc$-conjecture there are only finitely many integers $n$ for which 
\begin{equation} \label{ww}
\lambda_m(n,k) \leq n^{1 - \frac{m}{km-k} - \varepsilon}.
\end{equation}
This is a very good counterpart to their unconditional results with exponent $1-m/(km-1)$ and $1-m/(km-2)$ mentioned in Section 6.
Furthermore they proved that, under the $abc$-conjecture, for any $\varepsilon > 0$,
\begin{equation} \label{ww2}
Q_m(n,k) \gg_{k,m, \varepsilon} n^{k - 1 - \frac {1} {m-1} -\varepsilon}
\end{equation}
(\cite{ww}, Proposition 3.3).
Note that, apart from $\varepsilon$, both (\ref{lan}) and (\ref{ww2}) for $k=2$ are the best possible. For (\ref{lan}) it suffices to observe that $n$ can be a power of 2 and the other terms are at most $n+k$, for (\ref{ww2}) that $x=n+1, y=n$ can be a solution of the Pell equation $x^2 - 2y^2=1$.

It is easy to see that \cite{ww} is not far from the best possible.
There are infinitely many cases in which 

\begin{equation} \label{qmplus}
Q_m(n,k) \ll_{k,m} n^{k-1-\frac1m}.
\end{equation}
Indeed, let $m$ be odd, $a$ an $m$-th power and consider $n=a^m+1$. Then the $m$-th power part of $N$ is at least $a^m(a+1) \geq n^{1+1/m}$. This implies \eqref{qmplus}.

We shall prove (\ref{lan}),(\ref{ww2}) and a related result by a new method based on the following lemma which may have some independent interest.

\begin{lemma}
Let $k$ be a positive integer. Let $x$ be a positive real variable.
Then 
$$ \prod_{{\rm even}~i=0}^k (x+i)^{k\choose i} - \prod_{{\rm odd}~i=0}^k (x+i)^{k\choose i} = 
-(k-1)! x^{2^{k-1} -k} +O_k(x^{2^{k-1}-k-1}). $$
\end{lemma}

\begin{proof}
The following formula holds for Stirling numbers of the second kind:
$$ S_{n,k} = \frac {1}{k!} \sum_{i=0}^k (-1)^{k-i} {k\choose i} i^n.$$
We have $S(n,n) =1$ and  $S(n,k) = 0$ for $k>n$ (see e.g. \cite{as}, pp. 824-825).
It follows that, as $ x \to 0$,
$$ \sum_{i=0}^k (-1)^i {k \choose i} \log (1+ix)  = \sum_{i=0}^k (-1)^i {k \choose i} \sum_{j=1}^{\infty} (-1)^{j-1} \frac {i^j x^j}{j}=$$
$$ \sum_{j=1}^{\infty} \frac {(-1)^{k-j-1} x^j}{j} \sum_{i=0}^k (-1)^{k-i} {k \choose i} i^j =  -(k-1)! x^k +O_k(x^{k+1}). $$
Hence
$$\frac { \prod_{i~{\rm even},~i=0}^k~ (1+ix)^{k\choose i}} { \prod_{i~{\rm odd},~i=0}^k~ (1+ix)^{k\choose i}} -1 = 
-(k-1)! x^k +O_k(x^{k+1}). $$
Since $\sum_{i=0}^k (-1)^i {k \choose i} = (1-1)^k =0$, we obtain, for $ x \to \infty$,
$$\frac {\prod_{i~{\rm even},~i=0}^k~ (x+i)^{k\choose i}} { \prod_{i~{\rm odd},~i=0}^k~ (x+i)^{k\choose i}} -1 = 
-(k-1)! x^{-k} +O_k(x^{-k-1}). $$
Thus, using that $\sum_{i~{\rm odd},~i=0}^k~ {k \choose i} = 2^{k-1}$,
$$ \prod_{i~{\rm even,}~i=0}^k~ (x+i)^{k\choose i} - \prod_{i~{\rm odd},~i=0}^k~ (x+i)^{k\choose i} = 
-(k-1)! x^{2^{k-1} -k} +O_k(x^{2^{k-1}-k-1}). $$
\end{proof}

\begin{theorem} Let $k,m$ and $n$ be integers with $n \geq k\geq 2, m \geq 2$. 
Then, assuming the $abc$-conjecture, for every $\varepsilon > 0$ we have\\
{\rm(a)}  $Q_m(n,k) \gg_{\varepsilon, k, m}  n^{k-1-\frac {1}{m-1} -\varepsilon}$, \\
{\rm(b)} $R(n,k) \gg_{\varepsilon,k} n^{k-1 -\varepsilon}$, \\
{\rm (c)} $P(n,k) \geq (k-1+o_k(1)) \log n$.
\end{theorem}

\begin{proof} (a) We apply the $abc$-conjecture to 
$$ \frac 1g  \prod_{{\rm even}~i=0}^{k-1} (n+i)^{k-1 \choose i} - \frac 1g \prod_{{\rm odd}~i=0}^{k-1} (n+i)^{k-1 \choose i} = 
O_k(n^{2^{k-2} -k+1}), $$
where $g$ is the greatest common divisor of the first two terms. By the argument (\ref{erdos}) of Erd\H{o}s, $g = O_k (1).$
Hence, by $AB^m =N \ll_k n^{k}$,
\begin{equation} \label{basic}
n^{2^{k-2}} \ll_{\varepsilon,k} \left(R(AB)  \cdot n^{2^{k-2}-k+1}\right)^{1+\varepsilon} \ll_{\varepsilon,k} \left(A^{\frac {m-1}{m}}n^{2^{k-2}-k+1+\frac{k}{m}}\right)^{1 + \varepsilon}. 
\end{equation}
It follows that
$$ A \gg_{\varepsilon,k,m} n^{k -1- \frac{1}{m-1} - \varepsilon \frac {m}{m-1}2^{k-2}}. $$
(b) It follows from (\ref{basic}) that 
$$R(AB) \gg_{\varepsilon,k} n^{k-1 - \varepsilon \cdot 2^{k-2}}.$$
(c) Since the product of the first $l$ primes is $e^{(1+o(1))l}$, we deduce from part (b) that 
$P(n,k) \geq (k-1+o_k(1)) \log n$.
\end{proof}

\subsection{Estimates for general $n,k$}

\subsubsection{Greatest prime factor}

\begin{theorem}
Let $0 < \varepsilon < 1/2$ and $n > k^{3/2}$. Then under the $abc$-conjecture, there exists a number $k_1$ depending only on $\varepsilon$ such that for $k \geq k_1$ we have $$ P(n,k) \geq (\frac 12 - \varepsilon) k\log n.$$
\end{theorem}

\begin{proof}
Let $k \geq k_1$. Assume that 
\begin{equation} \label{ineq7} 
P(n,k) < \left( \frac 12 - \varepsilon \right) k \log n.
\end{equation}
Then, for a suitable constant $c_{10},$ by the Prime Number Theorem with remainder term,
\begin{equation} \label{ineq8}
\omega(n,k) \leq \frac{(\frac12 - \varepsilon)k \log n}{ \log k + \log_2 n} + \frac{c_{10} k \log n}{(\log k + \log_2 n)^2} \leq \frac {(\frac 12 - \frac {\varepsilon}{2}) k \log n}{\log k + \log_2n}.
\end{equation}
For $0\leq i < k$, write $$n+i = A_iB_i$$ with $P(A_i) \leq k$ and gcd$(B_i, \prod_{p \leq k}p) = 1$. 
For every prime $p \leq k$, let $i_p$ with $0 \leq i_p < k$ be such that $p$ appears to the maximum power in $A_{i_p}$. Let $S_1$ be the set obtained from $\{0,1, \dots, k-1\}$ by deleting $i_p$ for each prime $p \leq k$. Then $|S_i| \geq k - \pi(k)$. By the combinatorial argument (\ref{erdos}) of Erd\H{o}s,
\begin{equation} \label{ineq9}
 \prod_{A_i \in S_1} A_i \leq k^k.
 \end{equation}
Suppose that there exist at most $(1- \frac{\varepsilon}{4})k$ number of $i$'s in $S_1$ such that 
$$ A_i \leq k^{8 / \varepsilon}.$$
Then there are at least $\varepsilon k/8$ number of $i$'s with $A_i > k^{8 / \varepsilon}.$ Consequently, from (\ref{ineq9}) we get
$$k^k \geq \prod_{i \in S_1} a_i > (k^{8 / \varepsilon})^{\varepsilon k/8} = k^k,$$
a contradiction.
Thus there exists a set $S_2 \subseteq S_1$ with $|S_2| > (1- \frac{\varepsilon}{4})k$ and 
$$A_i \leq k^{8/{\varepsilon}} ~~{\rm for}~~ i \in S_2.$$
Assume that there is at most one $B_i$ with $i \in S_2$ such that 
$$ \omega(B_i) \leq \frac {(\frac 12 - \frac{\varepsilon}{3}) \log n}{\log k + \log_2 n}.$$
Then 
$$\omega(n,k) \geq \sum_{i \in S_2} \omega (B_i) \geq \frac {(\frac 12 - \frac {\varepsilon}{3}) k\log n}{ \log k + \log_2 n} \left((1- \frac {\varepsilon}{4})k-1 \right).$$
From (\ref{ineq8}) we get
$$ \frac {(\frac 12 - \frac {11\varepsilon}{24}) k \log n}{ \log k + \log_2 n} \leq  \frac {(\frac 12 - \frac {\varepsilon}{2}) k \log n}{ \log k + \log_2 n}.$$
Thus there exist at least two $B_i$'s, say $B_{i_1}$ and $B_{i_2}$, with $i_1,i_2 \in S_2, i_1>i_2$ and 
$$\omega(B_{i_j}) \leq \frac {(\frac 12 - \frac {\varepsilon}{3}) \log n}{\log k + \log_2 n} ,~~ j=1,2.$$
By (\ref{ineq7}) we get
$$R(B_{i_j}) \leq (k \log n)^{\frac {(\frac 12 - \frac {\varepsilon}{3}) \log n}{\log k + \log_2 n}} \leq n^{\frac 12 - \frac {\varepsilon}{3}}.$$
Applying the abc-conjecture to
$$n+i_1 = n+i_2 +(i_1-i_2)$$
we get
$$n \leq c_{11}(R((n+i_1)(n+i_2))k)^{1 + \frac {\varepsilon}{2}} \leq c_{11}(R(A_{i_1}B_{i_1}A_{i_2}B_{i_2})k)^{1 + \frac {\varepsilon}{2}}$$
$$\leq c_{11}(k^{\frac{17}{\varepsilon}}n^{(1 - \frac{2\varepsilon}{3})})^{1+ \frac{\varepsilon}{2}} \leq c_{11} k^{\frac{17}{\varepsilon}(1+ \frac{\epsilon}{2})}n^{1-\frac{\varepsilon}{6}},$$
where $c_{11}$ is a constant depending only on $\varepsilon$. Therefore,
\begin{equation} \label{shor}
\log n \leq \frac{6}{\varepsilon} \log c_{11} + \frac {102}{\varepsilon^2}\left(1 + \frac {\varepsilon}{2}\right) \log k \leq c_{12} \log k
\end{equation}
where $c_{12}$ depends only on $\varepsilon$. By (\ref{sh74}) and (\ref{shor}) we get
$$P(n,k) \gg k \log k \frac{\log_2 k}{\log_3 k} \gg k \log n \frac{\log_2 n}{\log_3 n}.$$
Hence the assertion follows. 
\end{proof}

\section{Proof of the Erd\H{o}s-Woods conjecture under the explicit $abc$-conjecture}
As a special case of a result on the analogue of the Erd\H{o}s-Woods conjecture for arithmetic progressions it was proved in \cite{blsw} that under assumption of the $abc$-conjecture the original Erd\H{o}s-Woods conjecture admits only finitely many exceptions for every $k>2$ . (See \cite{blsw}, Proposition 1.) For $k=2$ there exist infinitely many solutions: for every integer $h \geq 2$ the pairs $(2^h-2, 2^h(2^h-2))$ and $ (2^h-1, (2^h-1)^2)$ have the same prime divisors. 

An explicit version of the $abc$-conjecture is the following.
\begin{conj} [Baker, \cite{bak}] \label{abcexpl}
Let $a,b$ and $c$ be pairwise coprime positive integers satisfying $a+b=c$. Then
\begin{equation} \label{bak}
c < \frac 65 R(abc) \frac {(\log R(abc))^{\omega(abc)}} {\omega(abc)!}.
\end{equation}
\end{conj}

Laishram and Shorey \cite{ls12} proved that a consequence of (\ref{bak}) is that
\begin{equation} \label{ls}
c< (R(abc))^{7/4}.
\end{equation}

Here we prove the complete Erd\H{o}s-Woods conjecture for $k \geq3$ under assumption of the explicit $abc$-conjecture. Of course, it suffices to prove it for $k=3$.
\begin{theorem}
Assume Conjecture \ref{abcexpl}. Then there are no positive integers $n_1<n_2$ such that for $i=0,1,2$ the numbers $n_1+i$ and $n_2+i$ have the same prime divisors.
\end{theorem}

\begin{proof}
Apply Conjecture \ref{abcexpl} with (\ref{ls}) in place of (\ref{bak}) to the equation $(n_2 +1)^2 - n_2(n_2+2) =1$. This yields $$ n_2^2 < (R(n_2(n_2+1)(n_2+2)))^{1.75}.$$
Observe that every prime $p$ dividing $n_2(n_2+1)(n_2+2)$ divides either $n_2-n_1$ or $(n_2+1)-(n_1+1)$ or $(n_2+2)-(n_1+2)$. Thus 
$$ R(n_2(n_2+1)(n_2+2)) \leq n_2 - n_1 < n_2.$$
Combining the inequalities we obtain
$$n_2^2 < n_2^{1.75},$$
a contradiction.
\end{proof}

\section{Generalizations}
More general results have been obtained for the greatest prime factor of the product of the terms of a finite arithmetical progression (see e.g. \cite{ls06}, \cite{st90}) and for the greatest prime factor of the value of a polynomial or binary form (see e.g. \cite{spts}), for the number of distinct prime divisors of the product of the terms of a finite arithmetical progression (see e.g. \cite{lai}) and of a polynomial value (see e.g. \cite{tu80}), for the greatest $m$-free part of a polynomial value (see \cite{tu82}) and for the analogue of the Erd\H{o}s-Woods conjecture for an arithmetical progression (see \cite{blsw}).

There are also more general results under the assumption that the $abc$-conjecture is true: for lower bounds for the greatest squarefree divisor of a polynomial value and a binary form value, see \cite{elk}, \cite{lanc}, \cite{gra}, for a lower bound for the greatest $m$-free part of a binary form value, see \cite{lanc}, and for the Erd\H{o}s-Woods conjecture for an arithmetical progression, see \cite{blsw}. 

Finally, for an estimate for the number of prime divisors of a product of consecutive integers under the assumption that Schinzel's Hypothesis H is true, see \cite{blst}.

\end{document}